\documentclass[11pt]{amsart}
\usepackage{amssymb}
\usepackage{amsthm}
\usepackage{amsmath}
\usepackage[dvips]{epsfig}
\usepackage{graphicx}
\usepackage{hyperref}
\usepackage{color}
\usepackage{url}
\usepackage{tikz}
\usepackage{caption}
\usepackage{subcaption}
\usepackage{epstopdf}
\usepackage{mathtools}
\usepackage[arrow, matrix, curve]{xy}
\usepackage{marginnote}
\usetikzlibrary{decorations.markings}

\makeatletter\@addtoreset {equation}{section}\makeatother

\newtheorem{theorem}{Theorem}

\newtheorem{lemma}{Lemma}
\theoremstyle{remark}
\newtheorem{remark}{Remark}
\theoremstyle{definition}

\theoremstyle{corollary}

{\begin{trivlist} \item[]{\em Proof }}%
{\hspace*{\fill}$\rule{.3\baselineskip}{.35\baselineskip}$\end{trivlist}}

\begin{document}

\title[Instability of $H^1$-stable peakons]{\bf Instability of $H^1$-stable peakons \\ in the Camassa--Holm equation}

\author{F\'{a}bio Natali}
\address[F. Natali]{Departamento de Matem\'{a}tica - Universidade Estadual de Maring\'{a}, Avenida Colombo 5790, CEP 87020-900, Maring\'{a}, PR, Brazil}
\email{fmanatali@uem.br}

\author{Dmitry E. Pelinovsky}
\address[D. Pelinovsky]{Department of Mathematics and Statistics, McMaster University,
Hamilton, Ontario, Canada, L8S 4K1}
\email{dmpeli@math.mcmaster.ca}


\begin{abstract}
It is well-known that peakons in the Camassa--Holm equation are $H^1$-orbitally stable
thanks to the presence of conserved quantities and properties of peakons as constrained energy minimizers.
By using the method of characteristics, we prove that piecewise $C^1$ perturbations to peakons
grow in time in spite of their stability in the $H^1$-norm. We also show that the linearized stability analysis
near peakons contradicts the $H^1$-orbital stability result, hence passage from linear to nonlinear theory
is false in $H^1$.
\end{abstract}

\date{\today}
\maketitle

\section{Introduction}

The Camassa-Holm (CH) equation \cite{CH}
\begin{equation}
\label{CH}
u_t - u_{txx} + 3 u u_x = 2 u_x u_{xx} + u u_{xxx}, \quad x \in \mathbb{R}
\end{equation}
can be rewritten in the convolution form
\begin{equation}
\label{CHconv}
u_t + u u_x + \frac{1}{2} \varphi' \ast \left( u^2 + \frac{1}{2} u_x^2 \right) = 0, \quad x \in \mathbb{R},
\end{equation}
where $\varphi(x) = e^{-|x|}$ is the Green function satisfying $(1 - \partial_x^2) \varphi = 2 \delta$
with $\delta$ being Dirac delta distribution and $(f \ast g)(x) := \int_{\mathbb{R}} f(x-y) g(y) dy$ being the convolution operator.
It is clear that $\varphi \in H^s$ for $s < \frac{3}{2}$, where $H^s := H^s(\mathbb{R})$ is the standard Sobolev space
of squared integrable distributions equipped with the norm $\| f \|_{H^s} := \| \langle \cdot \rangle^s \hat{f} \|_{L^2}$
with $\langle x \rangle := \sqrt{1 + x^2}$ and $\hat{f}$ being the Fourier transform of $f$ on $\mathbb{R}$.

The purpose of function $\varphi$ is not only to rewrite the evolution problem for the CH equation (\ref{CH})
in the convolution form (\ref{CHconv}) that depends on the first derivative of $u$ in $x$ and does not depend
on its higher-order derivatives. In addition,
$\varphi$ expresses a particular family of solutions $u(t,x) = c \varphi(x-ct)$ with $c \in \mathbb{R}$
which are referred to as {\em peakons} (or {\em peaked solitary waves}). Indeed, the validity of
these peakons as solutions to the CH equation in the convolution form (\ref{CHconv}) can be checked
directly from the identity
\begin{equation}
- \varphi + \frac{1}{2} \varphi^2 + \frac{3}{4} \varphi \ast \varphi^2 = 0, \quad x \in \mathbb{R},
\end{equation}
which is piecewise $C^1$ on both sides from the peak at $x = 0$.

Cauchy problem for the CH equation (\ref{CH}) with the initial data $u_0 \in H^3$
was studied in the series of papers \cite{C1,CE1,CE2}. It was shown in Theorem 4.1 in \cite{C1} that
if $y_0 := (1 - \partial_x^2) u_0$ does not change
sign on $\mathbb{R}$, then the corresponding solution $u \in C(\mathbb{R}^+,H^3)$
exists globally in time. If there exists $x_0 \in \mathbb{R}$ such that
$(x-x_0) y_0(x) \geq 0$ on $\mathbb{R}$, the same conclusion applies (Theorem 4.4 in \cite{C1}),
whereas if $(x-x_0) y_0(x) \leq 0$ on $\mathbb{R}$, the local solution breaks in a finite time
in the sense that the slope of the solution becomes unbounded from below
in a finite time (Theorems 5.1 and 5.2 in \cite{C1}).

The condition $H^3$ on the initial data was used in \cite{C1,CE1,CE2}
to control the auxiliary quantity $y_0$ and to extend local solutions to global
solutions of the CH equation (\ref{CH}). Without these requirements, local well-posedness
of the Cauchy problem can be proven for initial data in $H^s$ for every $s > \frac{3}{2}$ \cite{LO,RB}
but cannot be pushed below and at $s = \frac{3}{2}$ because of lack of uniform
continuity of the local solution with respect to initial data and the norm inflation \cite{B,Molinet}.

Cauchy problem for the CH equation in the convolution form (\ref{CHconv}) with the initial data
in $H^1$ was studied in \cite{BC1} (similar results appear also in \cite{Holden})
by means of a coordinate transformation of the quasilinear equation to an equivalent semilinear system.
It was proven that the Cauchy problem for the equivalent semilinear system admits a unique global solution
(Theorem 1 in \cite{BC1}), which provides a global conservative solution to the CH equation
(Theorem 2 in \cite{BC1}) such that
\begin{equation}
\label{conservative-solution}
\left\{ \begin{array}{ll} u(t,\cdot) \in H^1 & \mbox{\rm for every \;} t \in \mathbb{R}^+, \\
\| u(t,\cdot) \|_{H^1} = \| u_0 \|_{H^1} & \mbox{\rm for almost every \;} t \in \mathbb{R}^+.
\end{array} \right.
\end{equation}
This global conservative solution is consistent with the two conserved quantities of the CH equation:
\begin{equation}
\label{conserved-quantities}
E(u) = \int_{\mathbb{R}} (u^2 + u_x^2) dx, \quad F(u) = \int_{\mathbb{R}} u (u^2 + u_x^2) dx.
\end{equation}
Moreover, it was proven in \cite{BCZ} that the global conservative solution
to the CH equation (\ref{CHconv}) is unique for every initial data in $H^1$.
Continuous dependence from initial data and local well-posedness of the weak solutions
\begin{equation}
\label{LWP-peakons}
u \in C((-T,T),H^1) \cap L^{\infty}((-T,T),W^{1,\infty}) \cap C^1((-T,T),L^2)
\end{equation}
to the CH equation (\ref{CHconv}) was established very recently in \cite{Linares}
for every $u_0 \in H^1 \cap W^{1,\infty}$, where $W^{1,\infty}$ is the Sobolev space
of functions with bounded first derivatives and $T > 0$ is a local existence time.

The previous study of stability of peakon solutions in \cite{CM,CS} relies on the existence of
conserved quantities (\ref{conserved-quantities}). It was proven in \cite{CM} that
the peakon $\varphi$ is a unique (up to translation) minimizer of $E(u)$ in $H^1$ subject to the constraint
$3 F(u) = 2 E(u)$ (Proposition 3.6 in \cite{CM}), where $E(\varphi) = 2$ and $F(\varphi) = \frac{4}{3}$.
Consequently, global smooth solutions $u \in C(\mathbb{R}^+,H^3)$ which are close to $\varphi$ in $H^1$
remains close to the translated orbit $\{ \varphi(\cdot - a) \}_{a \in \mathbb{R}}$
in $H^1$ for all $t > 0$ (Theorem 3.1 in \cite{CM}).

A different result on stability of peakons was proven in \cite{CS}, which we reproduce here.

\begin{theorem}\cite{CS}
\label{theorem-CS}
Assume existence of a solution $u \in C((0,T),H^1)$ to the CH equation (\ref{CH})
with either finite or infinite $T > 0$. For every small $\varepsilon > 0$, if the initial data satisfies
\begin{equation}
\label{H1-initial}
\| u_0 - \varphi \|_{H^1} < \left(\frac{\varepsilon}{3} \right)^4,
\end{equation}
then the solution satisfies
\begin{equation}
\label{H1-final}
\| u(t,\cdot) - \varphi(\cdot - \xi(t)) \|_{H^1} < \varepsilon, \quad t \in (0,T),
\end{equation}
where $\xi(t)$ is a point of maximum for $u(t,\cdot)$.
\end{theorem}

In Theorem \ref{theorem-CS}, the local solution
$u \in C((0,T),H^1)$ to the CH equation (\ref{CH}) may break in a finite time $T < \infty$
in the sense of Theorem 5.1 in \cite{C1}:
\begin{equation}
\label{blow-up-criterion}
u_x(t,x) \to -\infty \quad \mbox{\rm at some} \;\; x \in \mathbb{R} \quad \mbox{\rm as} \;\; t \nearrow T.
\end{equation}
Nevertheless, the $H^1$-norm of the
solution $u \in C((0,T),H^1)$ remains finite as $t \nearrow T$ thanks to the energy conservation $E(u)$
in time $t$ up to the blowup time $T$.
This implies that the bound (\ref{H1-final}) remains valid in the limit $t \nearrow T$ and
allows us to say that {\em peakons of the CH equation are $H^1$ orbitally stable.}

Various extensions of the $H^1$ orbital stability of peakons have been made recently.
Orbital stability of peaked periodic waves in the CH equation was proven in \cite{Lenells2004b,Lenells2005d}.
Stability of peakons in another integrable equation called the
Degasperis--Procesi equation was established in \cite{LinLiu} by extending ideas of \cite{CS}.
Asymptotic stability of peakons in the class of $H^1$ functions $u$ with $y := (1-\partial_x^2) u$ being
a non-negative finite measure is proven in \cite{Molinet1}. Asymptotic stability of trains of peakons
and anti-peakons with $y$ being a sign-indefinite finite measure was constructed recently in \cite{Molinet2}.

Multi-peakon solutions were constructed by many analytical and numerical tools \cite{S1,Chertok,H1,Lund}.
The local characteristic curve $x = \xi(t)$ for the Camassa--Holm equation (\ref{CHconv}) is
defined by the equation
\begin{equation}
\label{characteristic-speed}
\frac{d \xi}{dt} = u(t,\xi(t)).
\end{equation}
Since the multipeakon solution is known in the closed form $u(t,x) = \sum_{j=1}^N m_j(t) \varphi(x-x_j(t))$,
where $\{ x_k(t), m_k(t) \}_{k=1}^N$ satisfies the finite-dimensional Hamiltonian system
$$
\frac{dx_k}{dt} = \sum_{j = 1}^N m_j \varphi(x_k - x_j), \quad \frac{d m_k}{dt} = -\sum_{j=1}^N m_j m_k \varphi'(x_k-x_j),
$$
it is clear that the single peakons move along the local characteristic curves \cite{Lund} so that
$$
\frac{dx_k}{dt} = u(t,x_k(t)), \quad k \in \{1,2,\dots,N\}.
$$
Global conservative solutions to the CH equation with multi-peakons were studied with the
inverse scattering transform method in \cite{Kostenko}, where the long-time behavior of solutions with peakons
was investigated (see also \cite{S1,Li}) and details of collisions between peakons and anti-peakons were given
(see also \cite{BC1,Molinet2}).

The purpose of this work is to address the question of stability of a single peakon
in the time evolution of the CH equation beyond the $H^1$ orbital stability result of Theorem \ref{theorem-CS}.
We start with the linearized stability analysis and study evolution of the linearized equation around the peakons
by using the method of characteristics. The same method is also useful to prove nonlinear instability of
piecewise $C^1$ perturbations to peakons. This instability develops in spite of the $H^1$ orbital stability
result of Theorem \ref{theorem-CS}.

The previous works avoid the question of linearized stability of peakons.
It was noticed in \cite{CM} that ``the nonlinearity plays a dominant role rather than being a higher-order correction",
so that ``the passage from the linear to the nonlinear theory is not an easy task, and may even be false".
The authors of \cite{LinLiu} added that the peakons are not differentiable in $x$, which makes
it difficult to analyze the spectrum of the linearized operator around the peakons.

By adding a perturbation $v(t,x-t)$ to the single peakon $\varphi(x-t)$ moving with the normalized speed $c = 1$
and dropping the quadratic term in $v$, we obtain the linearized equation for $v(t,x-t)$
from the CH equation in the convolution form (\ref{CHconv}):
\begin{equation}
\label{linCH}
v_t + (\varphi - 1) v_x + \varphi' v + \varphi' \ast \left( \varphi v + \frac{1}{2} \varphi' v_x \right) = 0, \quad x - t \neq 0.
\end{equation}
In what follows, we use $x$ instead of $x-t$ thanks to the translational invariance
of the convolution operator. After the change, the location of the peak of $\varphi$ is placed at $x = 0$.
The following theorem represents the first main result of this work.

\begin{theorem}
\label{theorem-linear}
For every initial data $v_0 \in H^1$ satisfying $v_0(0) = 0$,
there exists a unique global solution $v \in C(\mathbb{R}^+,H^1)$ to the linearized equation
(\ref{linCH}) satisfying $v(t,0) = 0$,
\begin{equation}
\label{bound-1}
\| v(t,\cdot) \|_{H^1(0,\infty)}^2 = \| v_0 \|_{H^1(0,\infty)}^2 + 2 (e^t - 1) \int_0^{\infty} \varphi(s) \left( [v_0(s)]^2 + \frac{1}{2} [v'_0(s)]^2 \right) ds
\end{equation}
and
\begin{equation}
\label{bound-2}
\| v(t,\cdot) \|_{H^1(-\infty,0)}^2 = \| v_0 \|_{H^1(-\infty,0)}^2 +
2 (e^{-t} - 1) \int_{-\infty}^0 \varphi(s) \left( [v_0(s)]^2 + \frac{1}{2} [v'_0(s)]^2 \right) ds
\end{equation}
for every $t > 0$.
\end{theorem}

In Theorem \ref{theorem-linear}, we confirm the expectation from \cite{CS} that
the passage from the linear to the nonlinear theory may be false in $H^1$.
Indeed, the sharp exponential growth of $\| v(t,\cdot) \|_{H^1(0,\infty)}$
in (\ref{bound-1}) for the solution of the linearized equation (\ref{linCH})
contradicts the bound (\ref{H1-final}) in Theorem \ref{theorem-CS}
obtained for the solution of the full nonlinear equation (\ref{CHconv}).

On the other hand, by solving the linearized equation (\ref{linCH}) with a
method of characteristics, we discover the intrinsic instability associated
with the peaked profile of the traveling wave $\varphi$ (Lemma \ref{lem-2}). This instability is related
to the characteristics to the right of the peak for $x > 0$ but not to the left of the peak for $x < 0$.

Within the linearized equation (\ref{linCH}), we also discover that if $v_0(0) \neq 0$,
the continuous initial data $v_0 \in H^1$ generates a finite jump discontinuity in the solution $v(t,x)$
at the peak $x = 0$ for every small $t > 0$ (Lemma \ref{lem-1}).
This finite jump discontinuity is allowed in the domain of the linearized operator
associated with the linearized equation (\ref{linCH}) in $L^2$ (Remark \ref{rem-1}),
however it prevents the solution $v(t,\cdot)$ to stay in $H^1$ for every $t > 0$.
Related to this fact, we prove that the single peak of a perturbed peakon in the CH equation
in the convolution form (\ref{CHconv}) moves with the speed equal to the local characteristic speed as in (\ref{characteristic-speed})
(Lemma \ref{lem-nonlinear-2}). This allows us to show that the constraint on the solution $v$ at the peak required in Theorem \ref{theorem-linear}
is satisfied identically in the time evolution of the full nonlinear equation (\ref{CHconv}) (Remark \ref{rem-3}).

Similarly, we show for the linearized equation (\ref{linCH}) that if $v_0 \in H^1 \cap C^1$
with $v_0(0) = 0$ and $v_0'(0) \neq 0$, then $v(t,\cdot) \notin C^1$ for every small $t > 0$
because of the finite jump discontinuity of $v_x(t,x)$ at $x = 0$ for $t > 0$ (Remark \ref{rem-2}).
Compared to the constraint on $v$ at the peak, there is no way to
maintain a constraint on $v_x$ at the peak and to prevent the finite jump discontinuity of $v_x(t,x)$ at $x = 0$
in the time evolution of the full nonlinear equation (\ref{CHconv}).

The above facts suggest that the $L^2$-based spaces like $H^1$ or $H^2$ may not be the best spaces
to study the intrinsic instability of peakons in the CH equation. Instead,
we work in the space of bounded continuous functions on $\mathbb{R}$ which are piecewise
continuous differentiable with a finite jump discontinuity of the first derivative at $x = 0$ and
bounded first derivatives away from $x = 0$. This space is denoted by $C^1_0$:
\begin{equation}
\label{C-1-0}
C^1_0 := \{ v \in C^0(\mathbb{R}) \cap C^1(\mathbb{R}^-) \cap C^1(\mathbb{R}^+); \quad v, \partial_x v \in L^{\infty}(\mathbb{R}) \}.
\end{equation}
The linear instability result of Theorem \ref{theorem-linear} is easily extended to
prove exponential instability of $\| v_x(t,\cdot) \|_{L^{\infty}}$ because
of the characteristics to the right of the peak (Lemma \ref{lem-3}).
Unlike the linear instability in $H^1$ of Theorem \ref{theorem-linear}, which cannot be true in the nonlinear
evolution due to the result of Theorem \ref{theorem-CS}, the linear instability in $C^1_0$
persists in the full nonlinear equation (\ref{CHconv}) as the nonlinear instability of peakons
with respect to piecewise $C^1$ perturbations (Lemma \ref{lem-nonlinear-5}). The following theorem
represents the second main result of this work.

\begin{theorem}
\label{theorem-nonlinear}
For every $\delta > 0$, there exist $t_0 > 0$ and $u_0 \in H^1 \cap C^1_0$
satisfying
\begin{equation}
\label{initial-bound-theorem}
\| u_0 - \varphi \|_{H^1} + \| u_0' - \varphi' \|_{L^{\infty}} < \delta,
\end{equation}
such that the global conservative solution to the CH equation (\ref{CHconv}) with the initial data $u_0$ satisfies
\begin{equation}
\label{final-bound-theorem}
\| u_x(t_0,\cdot) - \varphi'(\cdot - \xi(t_0)) \|_{L^{\infty}} > 1,
\end{equation}
where $\xi(t)$ is a point of peak of $u(t,\cdot + \xi(t)) \in H^1 \cap C^1_0$ for $t \in [0,t_0]$.
\end{theorem}

We note that the peakon $\varphi$ is located on the boundary between global and breaking solutions
to the Camassa--Holm equation in the sense that $(1 - \partial_x^2) \varphi = 2 \delta$
is zero everywhere except at the peak. It is expected that some perturbations to the peakons
lead to global solutions whereas some others break in a finite time in the same sense
as (\ref{blow-up-criterion}). The nonlinear instability result in Theorem \ref{theorem-nonlinear}
does not distinguish between these two possible scenarios because the bound (\ref{final-bound-theorem})
is attained before the blowup time $T$. Note that we measure
the instability of peakons in the same norm as the one used to study wave breaking
in the Camassa--Holm equation in \cite{C1}. Although perturbations
to the peakons are not smooth in $H^1 \cap C^1_0$, we have justified the same
blow-up criterion in the method of characteristics for the solutions in $H^1 \cap C^1_0$
as the criterion (\ref{blow-up-criterion}) for the smooth solutions (Lemma \ref{lem-nonlinear-3}).

A general global conservative solution to the CH equation (\ref{CHconv}) satisfies
(\ref{conservative-solution}) and may have finite jumps of the energy $E(u)$ in (\ref{conserved-quantities})
at some time instances, which correspond to collision between peakons and anti-peakons \cite{BC1}.
The initial data $u_0 \in H^1 \cap C^1_0$ in Theorem \ref{theorem-nonlinear} exclude anti-peakons
and hence no jumps of the energy $E(u)$ occur in the time evolution
of the global conservative solution, which hence satisfies $u \in C(\mathbb{R}^+,H^1)$.
Since $u_0 \in H^1 \cap W^{1,\infty}$, this unique global conservative solution coincides
for $t \in (0,T)$ with the local solution (\ref{LWP-peakons}) constructed in \cite{Linares} where the maximal existence
time $T > 0$ may be finite because of the blow-up of the $W^{1,\infty}$ norm of the local solution.
Again, this blow-up of the $W^{1,\infty}$ norm agrees well with the instability criterion (\ref{final-bound-theorem}) in Theorem \ref{theorem-nonlinear}.

Study of the instability of peakons in the CH equation is inspired by
the recent work \cite{EP,GP1,GP2} on smooth and peaked periodic waves in
the reduced Ostrovsky equation,
\begin{equation}
\label{redOst}
u_t + u u_x = \partial_x^{-1} u,
\end{equation}
which is another generalization of the inviscid Burgers equation $u_t + u u_x = 0$.
While smooth periodic waves are linearly and nonlinearly stable \cite{EP,GP1},
peaked periodic waves were found to be linearly unstable because
the $L^2$-norm of the perturbation grows exponentially in time \cite{GP2}.
The linear instability was found from the solution of the truncated linearized equation
obtained by the method of characteristics and from the estimates on the solution of the full
linearized equation. Nonlinear instability was not studied in \cite{GP2}
due to the lack of global well-posedness results on solutions of the reduced Ostrovsky equation (\ref{redOst}) in $H^1$.

Compared to \cite{GP2}, we show here that the full linearized equation (\ref{linCH}) can be solved by method of characteristics
without truncation and that the nonlinear instability of peakons in the CH equation (\ref{CHconv}) can be concluded from
the linear instability of perturbations in $H^1 \cap C^1_0$.

The remainder of the article is organized as follows. Linearized evolution near a single peakon
is studied in Section 2, where the proof of Theorem \ref{theorem-linear} is given. Nonlinear evolution of
piecewise $C^1$ perturbations to a single peakon is studied in Section 3, where
the proof of Theorem \ref{theorem-nonlinear} is given. Section 4 concludes the article
with ideas for further work.

\vspace{5cm}

\section{Linearized evolution near a single peakon}

Let us first simplify the linearized equation (\ref{linCH}) by using the following elementary result.

\begin{lemma}
\label{lem-0}
Assume that $v \in H^1$. Then, for every $x \in \mathbb{R}$,
\begin{equation}
\label{tech-comp}
\varphi' \ast \left( \varphi v + \frac{1}{2} \varphi' v_x \right)(x) =
- \varphi'(x) v(x) + \varphi'(x)  v(0) - \varphi(x) \int_0^x v(y) dy.
\end{equation}
\end{lemma}

\begin{proof}
Since integrals of absolutely integrable functions are continuous and since
$\varphi, \varphi', v, v_x \in L^2$, the map
$$
x \mapsto \varphi' \ast \left( \varphi v + \frac{1}{2} \varphi' v_x \right)
$$
is continuous for every $x \in \mathbb{R}$. Now, $H^1(\mathbb{R})$ is continuously
embedded into the space of bounded continuous functions on $\mathbb{R}$
decaying to zero at infinity and thus, $v \in C(\mathbb{R}) \cap L^{\infty}(\mathbb{R})$.
Integrating by parts yields the following explicit expression for every $x \in \mathbb{R}$:
\begin{eqnarray*}
\left( \varphi' \ast \varphi' v_x\right)(x)
& = & \left( \varphi'' \ast \varphi' v \right)(x) - \left( \varphi' \ast \varphi'' v \right)(x) \\
& = & \left( \varphi \ast \varphi' v \right)(x) - 2 \varphi'(x) v(x) - \left( \varphi' \ast \varphi v \right)(x) + 2\varphi'(x)  v(0),
\end{eqnarray*}
which simplifies the left-hand side of (\ref{tech-comp}) to the form:
$$
\varphi' \ast \left( \varphi v + \frac{1}{2} \varphi' v_x \right)(x) =
- \varphi'(x) v(x) + \varphi'(x)  v(0) +\frac{1}{2}
\left( \varphi \ast \varphi' v\right)(x) +
\frac{1}{2} \left( \varphi' \ast \varphi v\right)(x).
$$
Furthermore, we obtain
\begin{eqnarray*}
\frac{1}{2} \left( \varphi \ast \varphi' v\right)(x) + \frac{1}{2} \left( \varphi' \ast \varphi v\right)(x) & = &
-\frac{1}{2} \int_{\mathbb{R}} \varphi(x-y) \varphi(y) v(y) \left[ {\rm sign}(y) + {\rm sign}(x-y) \right] dy \\
& = & - \int_0^x \varphi(x-y) \varphi(y) v(y) dy \\
& = & - \varphi(x) \int_0^x v(y) dy,
\end{eqnarray*}
which completes the proof of (\ref{tech-comp}).
\end{proof}

The Cauchy problem for the linear equation (\ref{linCH}) can be written in the evolution form:
\begin{equation}
\label{linCH-equiv}
\left\{ \begin{array}{l}
\frac{dv}{dt} = Av, \quad t > 0, \\
v |_{t = 0} = v_0, \end{array} \right.
\end{equation}
where the linearization near the single peakon $\varphi$ for a perturbation $v$ in $H^1$ is given by
\begin{equation}
\label{lin-op}
(Av)(x) = \left[ 1 - \varphi(x) \right] v'(x) + \varphi(x) \int_0^x v(y) dy - v(0) \varphi'(x), \quad x \neq 0,
\end{equation}
thanks to Lemma \ref{lem-0} and the translational invariance.

In order to define a strong solution to the Cauchy problem (\ref{linCH-equiv}), we consider
the operator $A : {\rm dom}(A) \subset L^2(\mathbb{R}) \mapsto L^2(\mathbb{R})$ with the maximal domain
given by
\begin{equation}
\label{dom-A}
{\rm dom}(A) = \{ v \in L^2(\mathbb{R}) : \quad Av \in L^2(\mathbb{R}) \}.
\end{equation}
Since $\varphi(x) \to 0$ as $|x| \to \infty$ exponentially fast and $\varphi' \in L^2(\mathbb{R})$,
${\rm dom}(A)$ is equivalent to
\begin{equation*}
{\rm dom}(A) \equiv \{ v \in L^2(\mathbb{R}) : \quad (1-\varphi) v' \in L^2(\mathbb{R}), \quad |v(0)| < \infty \}.
\end{equation*}

\begin{remark}
\label{rem-1}
Since $\varphi(0) = 1$, $H^1$ is continuously embedded into ${\rm dom}(A)$ but it is not equivalent
to ${\rm dom}(A)$. In particular, if $v \in H^1(\mathbb{R}^+) \cap H^1(\mathbb{R}^-)$
with a finite jump discontinuity at $x = 0$
then $v \in {\rm dom}(A)$ but $v \notin H^1$.
However, the representation (\ref{tech-comp}) in Lemma \ref{lem-0}
does not hold for solutions in ${\rm dom}(A)$ with a finite jump discontinuity at $x = 0$.
\end{remark}

Let us consider the linearized Cauchy problem (\ref{linCH-equiv})
in the space $C^1_0$ defined by (\ref{C-1-0}). The following lemma shows
that unless $v_0 \in C^1_0$ satisfies the constraint $v_0(0) = 0$,
solutions to the Cauchy problem (\ref{linCH-equiv}) do not remain in $C^1_0$ for $t \neq 0$.

\begin{lemma}
\label{lem-1}
Assume that $v_0 \in {\rm dom}(A) \cap C^1_0$ with $v_0(0) \neq 0$ and that there exists
a strong solution $v \in C(\mathbb{R}^+,{\rm dom}(A))$ to the Cauchy problem (\ref{linCH-equiv}).
Then, $v(t,\cdot) \notin C^1_0$ for every $t \in (0,t_0)$ with $t_0 > 0$ sufficiently small.
\end{lemma}

\begin{proof}
Assume that $v(t,\cdot) \in C^1_0$ for $t \in [0,t_0]$ with some $t_0 > 0$ and obtain a contradiction.
If $v(t,\cdot) \in C^1_0$, then it follows from (\ref{linCH-equiv}) and (\ref{lin-op}) for every $t \in (0,t_0)$ that
$$
\lim_{x \to \pm 0} \frac{dv}{dt}(t,x) = \pm v(t,0),
$$
and since $v(0,0) = v_0(0) \neq 0$, we have $v(t,\cdot) \notin C(\mathbb{R})$
for every $t \in (0,t_0)$ with $t_0 > 0$ sufficiently small
because of the finite jump discontinuity at $x = 0$.
Consequently, $v(t,\cdot) \notin C^1_0$ for every $t \in (0,t_0)$.
\end{proof}

Due to the reason in Lemma \ref{lem-1}, we set $v_0(0) = 0$ for the initial condition $v_0 \in H^1 \cap C^1_0$
in the Cauchy problem (\ref{linCH-equiv}). Note that ${\rm dom}(A) \cap C^1_0 \equiv H^1 \cap C^1_0$.
If $v \in H^1 \cap C^1_0$ and $v(0) = 0$, then $(Av)(x)$ in (\ref{lin-op}) is
continuous at $x = 0$ and its definition can be extended for every $x \in \mathbb{R}$ by
\begin{equation}
\label{lin-op-0}
(A v)(x) = \left[ 1 - \varphi(x) \right] v'(x) + \varphi(x) \int_0^x v(y) dy, \quad x \in \mathbb{R}.
\end{equation}
The next result shows that the condition $v_0(0) = 0$ on $v_0 \in H^1 \cap C^1_0$
is not only necessary but also sufficient
for existence of the unique global solution in $H^1 \cap C^1_0$ to the Cauchy problem (\ref{linCH-equiv}) with (\ref{lin-op-0}).

\begin{lemma}
\label{lem-2}
Assume that $v_0 \in H^1 \cap C^1_0$ with $v_0(0) = 0$. There exists the unique global solution
$v \in C(\mathbb{R},H^1 \cap C^1_0)$ to the Cauchy problem (\ref{linCH-equiv}) with (\ref{lin-op-0}) satisfying
$v(t,0) = 0$ for every $t \in \mathbb{R}$.
\end{lemma}

\begin{proof}
We solve the evolution problem (\ref{linCH-equiv}) with (\ref{lin-op-0})
by using the method of characteristics piecewise for $x > 0$ and $x < 0$.
The family of characteristic curves $X(t,s)$ satisfying the initial-value problem
\begin{equation}
\label{char}
\left\{ \begin{array}{l}
\frac{dX}{dt} = \varphi(X)-1, \\
X |_{t=0} = s, \end{array} \right.
\end{equation}
are uniquely defined for every $s \in \mathbb{R}$
thanks to the Lipschitz continuity of $\varphi$ on $\mathbb{R}$. The peak location at $X = 0$
is the critical point which remains invariant under the time flow of the initial-value problem (\ref{char}).
For $s \in \mathbb{R}$ and $t \in \mathbb{R}$, we obtain the family of characteristic curves in the exact form
\begin{equation}
\label{char-sol}
X(t,s) = \left\{ \begin{array}{ll}
\log\left[ 1 + (e^s - 1) e^{-t} \right], \quad & s > 0, \\
0, \quad & s = 0, \\
-\log\left[ 1 + (e^{-s} - 1) e^{t} \right], \quad & s < 0,
\end{array} \right.
\end{equation}
with $\lim_{s \to 0^{\pm}} X(t,s) = 0$ for every $t \in \mathbb{R}$.
Let us define
\begin{equation}
\label{w-0}
w_0(x) := \int_0^x v_0(y) dy, \quad w(t,x) := \int_0^x v(t,y) dy,
\end{equation}
then $w_0 \in C^1(\mathbb{R})$ with $w_0(0) = w_0'(0) = 0$. We are looking for
the solution $w(t,\cdot) \in C^1(\mathbb{R})$ with $w(t,0) = w_x(t,0) = 0$
for every $t \in \mathbb{R}$. Substituting $v = w_x$ into $v_t = (1-\varphi) v_x + \varphi w$
yields
$$
w_{tx} + (\varphi - 1) w_{xx} - \varphi w = 0, \quad x \in \mathbb{R},
$$
which can be integrated as follows:
$$
w_t + (\varphi - 1) w_x - \varphi' w = {\rm const} = 0, \quad x \in \mathbb{R},
$$
where the integration constant is zero thanks to the boundary condition $w(t,0) = 0$
for every $t \in \mathbb{R}$. Along each characteristic curve satisfying (\ref{char}),
$W(t,s) := w(t,X(t,s))$ satisfies the initial-value problem:
\begin{equation}
\label{w-eq}
\left\{ \begin{array}{l}
\frac{dW}{dt} = \varphi'(X(t,s)) W, \\
W |_{t=0} = w_0(s), \end{array} \right.
\end{equation}
which can be solved uniquely in the exact form:
\begin{equation}
\label{w-sol}
W(t,s) = w_0(s) X_s(t,s), \quad
X_s(t,s) = \left\{ \begin{array}{ll}
\frac{1}{1 + (e^t - 1) e^{-s}}, \quad & s > 0, \\
\frac{1}{1 + (e^{-t} - 1) e^{s}}, \quad & s < 0,
\end{array} \right.
\end{equation}
with $\lim_{s \to 0^{\pm}} W(t,s) = 0$ for every $t \in \mathbb{R}$ thanks to $w_0(0) = 0$.
Note the different limits $\lim_{s \to 0^{\pm}} X_s(t,s) = e^{\mp t}$ if $t \neq 0$,
hence $X(t,\cdot) \in C^1_0$ if $t \neq 0$. Also note that
\begin{equation}
\label{property-1}
X_s(t,s) > 0, \quad \mbox{\rm for every } \; s \neq 0
\end{equation}
and
\begin{equation}
\label{property-2}
X_s(t,s) \to 1 \quad \mbox{\rm as } \; |s| \to \infty.
\end{equation}

Finally, we solve (\ref{linCH-equiv}) with (\ref{lin-op-0})
by using (\ref{char-sol}) and (\ref{w-sol}). Along each characteristic curve satisfying (\ref{char}),
$V(t,s) := v(t,X(t,s))$ satisfies the initial-value problem:
\begin{equation}
\label{v-eq}
\left\{ \begin{array}{l}
\frac{dV}{dt} = \varphi(X(t,s)) W(t,s), \\
V |_{t=0} = v_0(s), \end{array} \right.
\end{equation}
which can be solved uniquely in the exact form:
\begin{equation}
\label{v-sol}
V(t,s) = v_0(s) + w_0(s) Y(t,s), \quad
Y(t,s) = \left\{ \begin{array}{ll}
\frac{(e^t-1) e^{-s}}{1 + (e^t - 1) e^{-s}}, \quad & s > 0, \\
\frac{(1-e^{-t}) e^s}{1 + (e^{-t} - 1) e^{s}}, \quad & s < 0,
\end{array} \right.
\end{equation}
with $\lim_{s \to 0^{\pm}} V(t,s) = 0$ for every $t \in \mathbb{R}$ thanks to $w_0(0) = v_0(0) = 0$.
Furthermore, $V(t,s)$ is continuously differentiable in $s$ piecewise for $s > 0$ and $s < 0$
for every $t \in \mathbb{R}$ since $v_0 \in C^1_0$ and $w_0 \in C^1$.
Hence, $V(t,\cdot) \in C^1_0$ for every $t \in \mathbb{R}$.
Also, thanks to the properties (\ref{property-1}) and (\ref{property-2}),
we have $V(t,\cdot) \in H^1$ for every $t \in \mathbb{R}$ since
$v_0 \in H^1$ and $w_0 \varphi \in H^1$.

Finally, thanks again to the properties (\ref{property-1}), (\ref{property-2}), and $X(t,0) = 0$ for every $t \in \mathbb{R}$,
the change of coordinates $(t,s) \to (t,X)$ is
a $C^1_0$ invertible transformation so that the solution $v(t,\cdot) = V(t,s(t,\cdot))$ belongs to $H^1 \cap C^1_0$
and satisfies $v(t,0) = V(t,0) = 0$ for every $t \in \mathbb{R}$.
\end{proof}

By analyzing the exact solution $v(t,\cdot) \in H^1 \cap C^1_0$ of Lemma \ref{lem-2} in $t$,
we show that $v_x(t,\cdot)$ grows in $t$ due to characteristics with $x > 0$.
At the same time, if we add an additional constraint $v_0 \in L^1$ on the initial data,
we can also show that $v(t,\cdot)$ remains bounded in the supremum norm as $t \to \infty$.

\begin{lemma}
\label{lem-3}
Assume that $v_0 \in H^1 \cap C^1_0 \cap L^1$ with $v_0(0) = 0$ and $\alpha := \lim_{x \to 0^+} v_0'(x) \neq 0$.
Then, we have for every $t \geq 0$:
\begin{equation}
\label{bounds-lin-1}
\| v(t,\cdot) \|_{L^{\infty}(0,\infty)} \leq \| v_0 \|_{L^{\infty}(0,\infty)} + \| v_0 \|_{L^1(0,\infty)}, \quad
\| v(t,\cdot) \|_{L^{\infty}(-\infty,0)} \leq 2 \| v_0 \|_{L^{\infty}(-\infty,0)}
\end{equation}
and
\begin{equation}
\label{bounds-lin-2}
\| v_x(t,\cdot) \|_{L^{\infty}(0,\infty)} \geq \alpha e^t, \quad
\| v_x(t,\cdot) \|_{L^{\infty}(-\infty,0)} \leq \| v_0' \|_{L^{\infty}(-\infty,0)} + 2 \| v_0 \|_{L^{\infty}(-\infty,0)}.
\end{equation}
\end{lemma}

\begin{proof}
If $v_0 \in H^1 \cap C^1_0$ with $v_0(0) = 0$, we have $V(t,\cdot) \in H^1 \cap C^1_0$
for every $t \in \mathbb{R}$ by Lemma \ref{lem-2}.
If in addition $v_0 \in L^1$, then $w_0 \in L^{\infty}$ thanks to
the bound $\| w_0 \|_{L^{\infty}} \leq \| v_0 \|_{L^1}$.

We obtain by elementary computations:
\begin{equation}
\label{elementary}
\max_{t \geq 0, s \geq 0} \frac{(e^t-1) e^{-s}}{1 + (e^t - 1) e^{-s}} = 1, \quad
\max_{t \geq 0, s \leq 0} \frac{(e^{-t}-1) s e^s}{1 + (e^{-t} - 1) e^{s}} = 1.
\end{equation}
While the proof of the first equality is obvious, the proof of the second equality consists of
several steps. For fixed $s \leq 0$, the function
$$
Z(t,s) := \frac{(e^{-t}-1) s e^s}{1 + (e^{-t} - 1) e^{s}}
$$
is monotonically increasing in $t$ so that
$$
Z_{\rm max}(s) := \max_{t \geq 0} Z(t,s) = \lim_{t \to \infty} Z(t,s) = \frac{|s| e^{-|s|}}{1 - e^{-|s|}}
$$
Then, $Z_{\rm max}(s)$ is monotonically decreasing in $|s|$ attaining the maximal value $1$ as $|s| \to 0$.
It follows from (\ref{v-sol}) and (\ref{elementary}) that
\begin{equation}
\label{V-estimate-1}
|V(t,s)| \leq |v_0(s)| + |w_0(s)|, \qquad t \geq 0, \quad s \geq 0
\end{equation}
and
\begin{equation}
\label{V-estimate-2}
|V(t,s)| \leq |v_0(s)| + \left| \frac{w_0(s)}{s} \right|, \qquad t \geq 0, \quad s \leq 0,
\end{equation}
where it follows from (\ref{w-0}) that
$$
\left| \frac{w_0(s)}{s} \right| \leq \| v_0 \|_{L^{\infty}}.
$$
Bounds (\ref{bounds-lin-1}) follow from the estimates (\ref{V-estimate-1}) and (\ref{V-estimate-2}).

By using the chain rule $U(t,s) = V_s(t,s)/X_s(t,s)$ for
$U(t,s) := v_x(t,X(t,s))$, we obtain from (\ref{v-sol}) the exact solution for $s > 0$ and $t > 0$:
\begin{equation}
\label{u-sol}
U(t,s) = v_0'(s) \left[ 1 + (e^t - 1) e^{-s} \right]
+ v_0(s) (e^t - 1) e^{-s} - w_0(s) \frac{(e^t-1) e^{-s}}{1 + (e^t - 1) e^{-s}}.
\end{equation}
If $\alpha := \lim_{x \to 0^+} v_0'(x) \neq 0$, then
the first bound in (\ref{bounds-lin-2}) follows from the estimate
\begin{equation}
\label{u-sol-limits}
\| U(t,\cdot) \|_{L^{\infty}(0,\infty)} \geq \lim_{s \to 0^+} U(t,s)
= \alpha e^t, \quad t \geq 0,
\end{equation}
thanks to $w_0(0) = v_0(0) = 0$.

Similarly, we obtain from (\ref{v-sol}) the exact solution for $s < 0$ and $t > 0$:
\begin{equation}
\label{u-sol-neg}
U(t,s) = v_0'(s) \left[ 1 + (e^{-t} - 1) e^{s} \right]
+ v_0(s) (1 - e^{-t}) e^{s} + w_0(s) \frac{(1 - e^{-t}) e^{s}}{1 + (e^{-t} - 1) e^{s}}.
\end{equation}
By using the second equality in (\ref{elementary}), we obtain from (\ref{u-sol-neg})  that
\begin{equation}
\label{u-sol-limits-neg}
|U(t,s)| \leq |v_0'(s)| + |v_0(s)| + \left| \frac{w_0(s)}{s} \right|, \quad t \geq 0, \quad s \leq 0,
\end{equation}
which yields the second bound in (\ref{bounds-lin-2}).
\end{proof}

\begin{remark}
\label{rem-2}
If $v_0 \in C^1$ with $v_0(0) = 0$ and $v_0'(0) \neq 0$, then
$\lim_{s \to 0^{\pm}} U(t,s) = v_0'(0) e^{\pm t}$ so that $U(t,\cdot) \notin C^1$ for $t \neq 0$.
Consequently, $v(t,\cdot) \notin C^1$ for $t \neq 0$.
\end{remark}

The exponential growth of $v_x(t,\cdot)$ as $t \to \infty$ in Lemma \ref{lem-3} discovers
the linear instability of the peakon in the supremum norm on $v_x$.
The same instability is observed in the $H^1$ norm, as formulated in Theorem \ref{theorem-linear}.
Below we prove equalities (\ref{bound-1}) and (\ref{bound-2}) in Theorem \ref{theorem-linear}
for $v_0 \in H^1$ with $v_0(0) = 0$ without additional requirements $v_0 \in C^1_0$ and $v_0 \in L^1$.

\vspace{0.2cm}

{\em Proof of Theorem \ref{theorem-linear}.} By working on $[0,\infty)$,
wee use the exact solutions in (\ref{v-sol}) and (\ref{u-sol}), integrate by parts, and obtain
\begin{equation}\label{L2-norm-p}
\| v(t,\cdot) \|_{L^2(0,\infty)}^2 = \int_0^{\infty}\frac{[v_0(s)]^2}{1 + (e^t - 1) e^{-s}} d s +
\int_0^{\infty}\frac{[w_0(s)]^2 (e^t-1) e^{-s}}{[1 + (e^t - 1) e^{-s}]^3}ds
\end{equation}
and
\begin{eqnarray}
\nonumber
\| v_x(t,\cdot) \|_{L^2(0,\infty)}^2 & = & \int_0^{\infty} [v_0'(s)]^2 [1 + (e^t - 1) e^{-s}] d s
+ 2 \int_0^{\infty} [v_0(s)]^2 (e^t - 1) e^{-s} d s \\
\label{der-L2-norm-p}
& \phantom{t} &
+ \int_0^{\infty} \frac{[v_0(s)]^2 (e^t - 1) e^{-s}}{1 + (e^t - 1) e^{-s}} d s
- \int_0^{\infty}\frac{w_0^2(s)(e^t-1)e^{-s}}{[1 + (e^t - 1) e^{-s}]^3}ds.
\end{eqnarray}
In order to obtain (\ref{L2-norm-p}), we use (\ref{v-sol}) and the chain rule to write
$$
\| v(t,\cdot) \|_{L^2(0,\infty)}^2 = \int_0^{\infty}
\left[ v_0(s) + w_0(s) \frac{(e^t-1) e^{-s}}{1 + (e^t - 1) e^{-s}} \right]^2 \frac{ds}{1 + (e^t - 1) e^{-s}}.
$$
We expand the square and integrate the middle term by parts with the boundary conditions
$w_0(0) = 0$ and $\lim_{s \to \infty} [w_0(s)]^2 e^{-s} = 0$
thanks to the H\"older inequality:
\begin{equation*}
[w_0(x)]^2=\left(\int_0^x v_0(s)ds \right)^2\leq x||v_0||_{L^2(0,\infty)}^2,\ \ \ \ \ \ x>0.
\end{equation*}
As a result, straightforward computations yield (\ref{L2-norm-p}). The computations
for (\ref{der-L2-norm-p}) are similar but longer with three terms integrated by parts
with the boundary conditions $v_0(0) = 0$ and $\lim_{s \to \infty} v_0(s) = 0$.
Summing (\ref{L2-norm-p}) and (\ref{der-L2-norm-p}) yields (\ref{bound-1}).

Computations on $(-\infty,0]$  are similar from the exact solutions in (\ref{v-sol}) and (\ref{u-sol-neg}).
Integration by parts yield
\begin{equation}\label{L2-norm-p-neg}
\| v(t,\cdot) \|_{L^2(-\infty,0)}^2 = \int_{-\infty}^0 \frac{[v_0(s)]^2}{1 + (e^{-t} - 1) e^{s}} d s +
\int_{-\infty}^0 \frac{[w_0(s)]^2 (e^{-t}-1) e^{s}}{[1 + (e^{-t} - 1) e^{s}]^3}ds
\end{equation}
and
\begin{eqnarray}
\nonumber
\| v_x(t,\cdot) \|_{L^2(-\infty,0)}^2 & = & \int_{-\infty}^0 [v_0'(s)]^2 [1 + (e^{-t} - 1) e^{s}] d s
+ 2 \int_{-\infty}^0 [v_0(s)]^2 (e^{-t} - 1) e^{s} d s \\
\label{der-L2-norm-p-neg}
& \phantom{t} &
+ \int_{-\infty}^0 \frac{[v_0(s)]^2 (e^{-t} - 1) e^{s}}{1 + (e^{-t} - 1) e^{s}} d s
- \int_{-\infty}^0\frac{w_0^2(s)(e^{-t}-1)e^{s}}{[1 + (e^{-t} - 1) e^{s}]^3}ds.
\end{eqnarray}
Summing (\ref{L2-norm-p-neg}) and (\ref{der-L2-norm-p-neg}) yields (\ref{bound-2}).
$\Box$

\section{Nonlinear evolution near a single peakon}

Global existence and uniqueness of the conservative solution (\ref{conservative-solution}) to the Cauchy problem
for the CH equation in the convolution form (\ref{CHconv}) with initial data $u_0 \in H^1$
was proven in \cite{BC1,BCZ}. Continuous dependence of the conservative solution (\ref{LWP-peakons})
for initial data $u_0 \in H^1 \cap W^{1,\infty}$ was proven in \cite{Linares}. We rewrite the Cauchy problem in the form:
\begin{equation}
\label{CHevol}
\left\{
\begin{array}{l}
u_t + u u_x + Q[u] = 0, \quad t > 0,\\
u |_{t = 0} = u_0,
\end{array} \right.
\end{equation}
where
\begin{equation}
\label{Q-def}
Q[u](x) := \frac{1}{2} \int_{\mathbb{R}} \varphi'(x-y) \left( [u(y)]^2 + \frac{1}{2} [u'(y)]^2 \right) dy, \quad
x \in \mathbb{R}.
\end{equation}
The following lemma shows that $Q[u](x)$ is continuous for every $x \in \mathbb{R}$ if $u \in H^1$.

\begin{lemma}
\label{lem-nonlinear-1}
For every $u \in H^1$, we have $Q[u] \in C(\mathbb{R})$.
\end{lemma}

\begin{proof}
We can rewrite (\ref{Q-def}) in the explicit form:
\begin{equation*}
Q[u](x) = \frac{1}{2} e^{x} \int_{x}^{\infty} e^{-y} \left( [u(y)]^2 + \frac{1}{2} [u'(y)]^2 \right) dy -
\frac{1}{2} e^{-x} \int_{-\infty}^{x} e^{y} \left( [u(y)]^2 + \frac{1}{2} [u'(y)]^2 \right) dy.
\end{equation*}
Each integral is a continuous function for every $x \in \mathbb{R}$ since it is given by
an integral of the absolutely integrable function. Hence $Q[u]$ is continuous on $\mathbb{R}$.
\end{proof}

The unique global conservative solution (\ref{conservative-solution}) to the Cauchy problem (\ref{CHevol})
satisfies the weak formulation
\begin{equation}
\label{CHweak}
\int_{0}^{\infty} \int_{\mathbb{R}} \left( u \psi_t + \frac{1}{2} u^2 \psi_x - Q[u] \psi \right) dx dt
+ \int_{\mathbb{R}} u_0(x) \psi(0,x) dx = 0,
\end{equation}
where the equality is true for every test function $\psi \in C^1_c(\mathbb{R}^+ \times \mathbb{R})$.
We consider the class of solutions with a single peak placed at the point
$\xi(t)$ so that $u(t,\cdot + \xi(t)) \in C^1_0$ for every $t > 0$, where $C^1_0$ is defined by (\ref{C-1-0}).
The following lemma shows that the single peak moves with its local characteristic speed
as in (\ref{characteristic-speed}).

\begin{lemma}
\label{lem-nonlinear-2}
Assume that there exists $T > 0$ such that the weak global conservative solution (\ref{conservative-solution})
to the equation (\ref{CHweak}) satisfies $u(t,\cdot + \xi(t)) \in H^1 \cap C^1_0$ for every $t \in (0,T)$
with a single peak located at $x = \xi(t)$. Then, $\xi \in C^1(0,T)$ satisfies
\begin{equation}
\label{peak}
\frac{d \xi}{dt} = u(t,\xi(t)), \quad t \in (0,T).
\end{equation}
\end{lemma}

\begin{proof}
Integrating (\ref{CHweak}) by parts on $(-\infty,\xi(t))$ and $(\xi(t),\infty)$
and using the fact that $u(t,\cdot) \in C(\mathbb{R})$, we obtain
the following equations piecewise outside the peak's location:
\begin{equation*}
u_t(t,x) + u(t,x) u_x(t,x) + Q[u](t,x) = 0, \quad \pm \left[ x-\xi(t) \right] > 0.
\end{equation*}
Since $u(t,\cdot) \in H^1$ for every $t \in \mathbb{R}^+$,
$Q[u]$ is a continuous function of $x$ on $\mathbb{R}$ for every $t \in \mathbb{R}^+$ by Lemma \ref{lem-nonlinear-1}.
Therefore, the function $u(t,\cdot + \xi(t)) \in C^1_0$ satisfies:
\begin{equation*}
[u_t]^+_- + u(t,\xi(t)) [u_x]^+_-  = 0, \quad t \in (0,T),
\end{equation*}
where $[u_x]^+_-$ is the jump of $u_x$ across the peak location. On the other hand,
since $u(t,\xi(t))$ is continuous and $u(t,\cdot + \xi(t)) \in C^1_0$, we differentiate $u(t,\xi(t))$
continuously on both sides from $x = \xi(t)$ and obtain
\begin{equation*}
[u_t]^+_- + \frac{d \xi}{dt} [u_x]^+_- = 0, \quad t \in (0,T).
\end{equation*}
Since $[u_x]^+_- \neq 0$ if $u \notin C^1(\mathbb{R})$, then $\xi(t)$ satisfies (\ref{peak})
and since $u \in C((0,T) \times \mathbb{R})$ due to Sobolev embedding of $H^1(\mathbb{R})$
into $C(\mathbb{R})$, then $\xi \in C^1(0,T)$.
\end{proof}

In order to study the nonlinear evolution near a single peakon, we decompose the weak global
conservative solution (\ref{conservative-solution}) as the following sum of the peakon and its small perturbation:
\begin{equation}
\label{decomp}
u(t,x) = \varphi(x-t-a(t)) + v(t,x-t-a(t)), \quad t \in \mathbb{R}^+, \quad x \in \mathbb{R},
\end{equation}
where $v(t,\cdot) \in H^1$ for every $t \in \mathbb{R}^+$ is the perturbation
and $a(t)$ is the deviation of the perturbed peakon's peak
from its unperturbed position moving with the unit speed.
If there exists $T > 0$ such that $v(t,\cdot) \in C^1_0$ for every $t \in (0,T)$, the global
conservative solution (\ref{conservative-solution}) satisfies $u(t,\cdot + \xi(t)) \in C^1_0$ for every $t \in (0,T)$,
that is, it has a single peak at $\xi(t) = t + a(t)$.
By Lemma \ref{lem-nonlinear-2}, $a \in C^1(0,T)$ satisfies the equation
\begin{equation}
\label{peak-moves}
\frac{da}{dt} = v(t,0), \quad t \in (0,T).
\end{equation}
Substituting (\ref{decomp}) and (\ref{peak-moves}) into the Cauchy problem (\ref{CHevol}) yields
the following Cauchy problem for the peaked perturbation $v$ to the peakon $\varphi$:
\begin{equation}
\label{CHpert}
\left\{
\begin{array}{l}
v_t = (1 - \varphi) v_x + \varphi w + (v|_{x=0} - v) v_x - Q[v], \quad t \in (0,T),\\
v |_{t = 0} = v_0,
\end{array} \right.
\end{equation}
where $w(t,x) = \int_0^x v(t,y) dy$ and the linear evolution has been simplified
by using (\ref{tech-comp}) in Lemma \ref{lem-0}. Note that $x - t - a(t)$
in (\ref{decomp}) becomes $x$ in (\ref{CHpert}) thanks to the translational invariance of the system
(\ref{CHevol}) with (\ref{Q-def}).

\begin{remark}
\label{rem-3}
The dynamical equation (\ref{peak-moves})
cancels the last term in the linearization at the peakon (\ref{lin-op})
without additional requirement of $v |_{x=0} = 0$ imposed in Lemmas \ref{lem-2} and \ref{lem-3}.
\end{remark}

Related to the Cauchy problem (\ref{CHpert}), we define the family of characteristic coordinates $X(t,s)$
satisfying the initial-value problem:
\begin{equation}
\label{char-X}
\left\{
\begin{array}{l}
\frac{dX}{dt} = \varphi(X) - 1 + v(t,X) - v(t,0), \quad t \in (0,T),\\
X |_{t = 0} = s,
\end{array} \right.
\end{equation}
Along each characteristic curve parameterized by $s$, let us define
$V(t,s) := v(t,X(t,s))$. It follows from (\ref{CHpert})
and (\ref{char-X}) that $V(t,s)$ on each characteristic curve $x = X(t,s)$ satisfies the initial-value problem:
\begin{equation}
\label{char-V}
\left\{
\begin{array}{l}
\frac{dV}{dt} = \varphi(X) w(t,X) - Q[v](X), \quad t \in (0,T),\\
V |_{t = 0} = v_0(s).
\end{array} \right.
\end{equation}
The following lemma transfers well-posedness theory for differential equations to
the existence, uniqueness, and smoothness of the family of characteristic coordinates
and the solution surface.

\begin{lemma}
\label{lem-nonlinear-3}
Assume $v_0 \in H^1 \cap C^1_0$. There exists $T > 0$ (finite or infinite) such that the unique family of
characteristic coordinates $[0,T) \times \mathbb{R} \ni (t,s) \mapsto X \in \mathbb{R}$ to (\ref{char-X})
and the unique solution surface $[0,T) \times \mathbb{R} \ni (t,s) \mapsto V \in \mathbb{R}$ to (\ref{char-V})
exist as long as $v_x(t,\cdot) \in L^{\infty}(\mathbb{R})$ for $t \in [0,T)$.
Moreover, $X$ and $V$ are $C^1$ in $t$ and $C^1_0$ in $s$ for every $(t,s) \in [0,T) \times \mathbb{R}$.
\end{lemma}

\begin{proof}
A simple extension of the proof of Lemma \ref{lem-nonlinear-1} implies that if $f \in H^1 \cap C^1_0$,
then $Q[f] \in C^1_0$. Every $C^1_0$ function is Lipschitz continuous at $x = 0$.
In addition, since $v(t,\cdot) \in H^1$ for every $t > 0$, then $v(t,\cdot) \in L^{\infty}(\mathbb{R})$, hence
the function $v(t,\cdot)$ is globally Lipschitz continuous when it is locally Lipschitz continuous.
Since $\varphi$ have the same properties on $\mathbb{R}$ and $w(t,\cdot) \in C^1(\mathbb{R})$ if $v(t,\cdot) \in C(\mathbb{R})$,
the right-hand-sides of systems (\ref{char-X}) and (\ref{char-V}) are global Lipschitz continuous functions of $X$ as long as the solution
$v(t,\cdot) \in H^1$ remains $v(t,\cdot) \in C^1_0$ for $t \in [0,T)$ with some (finite or infinite) $T > 0$.
Existence and uniqueness of the classical solutions $X(\cdot,s) \in C^1(0,T)$ and
$V(\cdot,s) \in C^1(0,T)$ for every $s \in \mathbb{R}$ follows from the ODE theory.
By the continuous dependence theorem, $X(t,\cdot) \in C^1_0$ and $V(t,\cdot) \in C^1_0$
for every $t \in [0,T)$.

Let us now show that $v(t,\cdot) \in H^1 \cap C^1_0$ for $t \in [0,T)$
if $v_x(t,\cdot) \in L^{\infty}(\mathbb{R})$ for $t \in [0,T)$. By differentiating
(\ref{char-X}) piecewise for $s > 0$ and $s < 0$, we obtain
\begin{equation}
\label{char-X-der}
\left\{
\begin{array}{l}
\frac{d X_s}{dt} = \left[ \varphi'(X) + v_x(t,X) \right] X_s, \quad t \in (0,T),\\
X_s |_{t = 0} = 1.
\end{array} \right.
\end{equation}
with the solution
$$
X_s(t,s) = \exp\left(\int_0^t \left[ \varphi'(X(t',s)) + v_x(t,X(t',s)) \right] dt' \right).
$$
If $v_x(t,\cdot) \in L^{\infty}(\mathbb{R})$ for $t \in [0,T)$, then $X_s(t,s) > 0$ for $t \in [0,T)$
piecewise for $s > 0$ and $s < 0$, hence the change of coordinates $(t,s) \to (t,X)$ is a $C^1_0$ invertible
transformation. As a result, $V(t,\cdot) \in C^1_0$ implies that $v(t,\cdot) \in C^1_0$ for $t \in [0,T)$.
\end{proof}

Since $\varphi(0) = 1$ and $v(t,\cdot) \in C(\mathbb{R})$ for every $t \in \mathbb{R}^+$,
$X = 0$ is a critical point of the initial-value problem (\ref{char-X}). Therefore,
the unique solution of Lemma \ref{lem-nonlinear-3} for $s = 0$ satisfies $X(t,0) = 0$. This limiting characteristic curve separates
the family of characteristic curves with $s > 0$ and $s < 0$. Since $V(t,s)$ is $C^1$ in $t$
for every $s \in \mathbb{R}$ and Lipschitz continuous at $s = 0$ by Lemma \ref{lem-nonlinear-3},
the limiting value $V_0(t) := V(t,0)$ is $C^1$ in $t$ and satisfies
\begin{equation}
\label{V-0}
\frac{d V_0}{dt} = - Q[v](0),
\end{equation}
where we have used $w(t,0) = 0$.

In order to control $v_x(t,\cdot) \in L^{\infty}(\mathbb{R})$ for $t \in [0,T)$
needed in the condition of Lemma \ref{lem-nonlinear-3}, we differentiate (\ref{CHpert}) in $x$ and obtain
\begin{equation}
\label{CHder}
\left\{
\begin{array}{l}
v_{xt} = (1 - \varphi) v_{xx} - \varphi' v_x + \varphi v + \varphi' w
+ (v|_{x=0} - v) v_{xx} - \frac{1}{2} v_x^2 + v^2 - P[v], \quad t \in (0,T),\\
v_x |_{t = 0} = v_0',
\end{array} \right.
\end{equation}
where we have used $(1 - \partial_x^2) \varphi = 2 \delta$ and have defined
\begin{equation}
\label{P-def}
P[v](x) := \frac{1}{2} \int_{\mathbb{R}} \varphi(x-y) \left( [v(y)]^2 + \frac{1}{2} [v'(y)]^2 \right) dy, \quad x \in \mathbb{R}.
\end{equation}
It follows from (\ref{char-X}) and (\ref{CHder}) that $U(t,s) := v_x(t,X(t,s))$ on each characteristic curve
satisfies the initial-value problem:
\begin{equation}
\label{char-U}
\left\{
\begin{array}{l}
\frac{dU}{dt} = -\varphi'(X) U + \varphi(X) V + \varphi'(X) w(t,X) - \frac{1}{2} U^2 + V^2 - P[v](X), \quad t \in (0,T),\\
U |_{t = 0} = v_0'(s).
\end{array} \right.
\end{equation}
Although $\varphi'(X)$ has a jump discontinuity at $X = 0$, the regions $\mathbb{R}^+$ and $\mathbb{R}^-$
for $s$ are separated from each other thanks to the fact that the limiting characteristic curve at $X = 0$
corresponds to the critical point of the initial-value problem (\ref{char-X}). As a result,
we consider the initial-value problem (\ref{char-U}) separately for $s > 0$ and $s < 0$.

\begin{lemma}
\label{lem-nonlinear-4}
Assume $v_0 \in H^1 \cap C^1_0$ and $T > 0$ (finite or infinite) be given by Lemma \ref{lem-nonlinear-3}.
There exist unique solutions $[0,T) \times \mathbb{R}^+ \ni (t,s) \mapsto U^+ \in \mathbb{R}$
and $[0,T) \times \mathbb{R}^- \ni (t,s) \mapsto U^- \in \mathbb{R}$ to (\ref{char-U})
as long as $v_x(t,\cdot) \in L^{\infty}(\mathbb{R})$ for $t \in [0,T)$.
Moreover, $U^{\pm}$ are $C^1$ in $t$ and $s$
for every $(t,s) \in [0,T) \times \mathbb{R}^{\pm}$.
\end{lemma}

\begin{proof}
Similarly to the proof of Lemma \ref{lem-nonlinear-1}, it follows that if $f \in H^1$,
then $P[f] \in C^1(\mathbb{R})$. In addition, $\varphi'(X(t,s))$ are $C^1$ functions of $t$ and $s$
for $t \in [0,T)$ and separately for $s \in \mathbb{R}^+$ and $s \in \mathbb{R}^-$.
Existence and uniqueness of solutions $[0,T) \times \mathbb{R}^+ \ni (t,s) \mapsto U^+ \in \mathbb{R}$
and $[0,T) \times \mathbb{R}^- \ni (t,s) \mapsto U^- \in \mathbb{R}$ to (\ref{char-U}) follow from the ODE
theory. By the continuous dependence theorem, $U^{\pm}$ are $C^1$ in $t$ and $s$
for every $(t,s) \in [0,T) \times \mathbb{R}^{\pm}$.
\end{proof}

By Lemma \ref{lem-nonlinear-4}, we are allowed to define the one-sided limits
$U_0^{\pm}(t) := \lim_{s \to 0^{\pm}} U^{\pm}(t,s)$ for $t \in [0,T)$.
The functions $U_0^{\pm}$ are  $C^1$ in $t$ and satisfy for $t \in (0,T)$:
\begin{equation}
\label{U-0}
\frac{d U_0^{\pm}}{dt} = \pm U_0^{\pm} + V_0 - \frac{1}{2} (U_0^{\pm})^2 + V_0^2 - P[v](0).
\end{equation}
By analyzing the time evolution (\ref{V-0}) and (\ref{U-0}), we finally prove
the nonlinear instability of peaked perturbations to the $H^1$-orbitally stable peakon.

\begin{lemma}
\label{lem-nonlinear-5}
For every $\delta > 0$, there exist $t_0 > 0$ and $v_0 \in H^1 \cap C^1_0$
satisfying
\begin{equation}
\label{initial-bound}
\| v_0 \|_{L^{\infty}} + \| v_0' \|_{L^{\infty}} < \delta,
\end{equation}
such that the unique solution $v(t,\cdot) \in H^1 \cap C^1_0$ to the Cauchy problem (\ref{CHpert})
in Lemmas \ref{lem-nonlinear-3} and \ref{lem-nonlinear-4} satisfies
\begin{equation}
\label{final-bound}
\| v_x(t_0,\cdot) \|_{L^{\infty}} > 1.
\end{equation}
\end{lemma}

\begin{proof}
Combining (\ref{V-0}) and (\ref{U-0}) together yields the following equation for $t \in (0,T)$:
\begin{equation}
\label{V-U-0}
\frac{d}{dt} (V_0 + U_0^+) = (V_0 + U_0^+) + V_0^2 - \frac{1}{2} (U_0^+)^2 + F_0,
\end{equation}
where
\begin{eqnarray*}
F_0 & := & - Q[v](0) - P[v](0) \\
& = & \frac{1}{2} \int_{-\infty}^{+\infty} \left[\varphi'(y) - \varphi(y) \right]
 \left( [v(t,y)]^2 + \frac{1}{2} [v_y(t,y)]^2 \right) dy, \\
& = & - \int_{0}^{+\infty} e^{-y} \left( [v(t,y)]^2 + \frac{1}{2} [v_y(t,y)]^2 \right) dy,
\end{eqnarray*}
and we have used $\varphi(-y) = \varphi(y)$, $\varphi'(-y) = - \varphi(y)$.
By using an integrating factor, we rewrite (\ref{V-U-0}) in the equivalent form
$$
\frac{d}{dt} \left[ e^{-t} (V_0 + U_0^+) \right] = e^{-t} \left[ V_0^2 - \frac{1}{2} (U_0^+)^2 + F_0 \right]
\leq e^{-t} V_0^2,
$$
where the last inequality is due to $F_0 \leq 0$. Integrating the differential inequality yields the bound:
\begin{equation}
\label{diff-ineq}
V_0(t) + U_0^+(t) \leq e^{t} \left[  V_0(0) + U_0^+(0) + \int_0^t e^{-\tau} V_0^2(\tau) d\tau \right], \quad t \in [0,T).
\end{equation}

Since $\varphi'(x) \to \mp 1$ as $x \to \pm 0$ and $\varphi(x) \to 0$ as $|x| \to \infty$,
whereas $v(t,\cdot)$ is small in the $C^1_0$ norm at least for $t \in [0,T)$,
the coordinate $\xi(t) = t + a(t)$ for the peak's location coincides with the location
of the maximum of $u(t,\cdot)$ in the decomposition (\ref{decomp}) for $t \in [0,T)$.
By Theorem \ref{theorem-CS}, for every small $\varepsilon > 0$,
if $\| v_0 \|_{H^1} < (\varepsilon/3)^4$, then
$$
\| v(t,\cdot) \|_{H^1} < \varepsilon, \quad t \in (0,T),
$$
By Sobolev's embedding, we have
\begin{equation}
\label{bound-on-V0}
|V_0(t)| \leq \| v(t,\cdot) \|_{L^{\infty}} \leq \frac{1}{\sqrt{2}} \| v(t,\cdot) \|_{H^1} < \varepsilon.
\end{equation}
Let us assume that the initial data $v_0 \in H^1 \cap C^1_0$ satisfies
$v_0(0) = 0$ and
\begin{equation}
\label{eq-alpha}
\alpha := \lim_{x \to 0^+} v_0'(x) = -\| v_0' \|_{L^{\infty}} = -2 \varepsilon^2.
\end{equation}
The initial bound (\ref{initial-bound}) is consistent with (\ref{eq-alpha}) if
for every small $\delta > 0$, the small value of $\varepsilon$ satisfies
$$
\left(\frac{\varepsilon}{3} \right)^4 + 2 \varepsilon^2 < \delta,
$$
which just specifies $\varepsilon$ in terms of $\delta$. With these constraints on $v_0$,
the bound (\ref{diff-ineq}) yields
\begin{equation*}
V_0(t) + U_0^+(t) \leq e^{t} \left[  \alpha + \varepsilon^2 \right] = - \varepsilon^2 e^{t},
\end{equation*}
or equivalently, $|V_0(t) + U_0^+(t)| \geq \varepsilon^2 e^{t}$. Hence, for every
small $\varepsilon > 0$ there exists sufficiently large
$$
\tau := \log(2) - 2 \log(\varepsilon) > 0,
$$
such that $|V_0(\tau) + U_0^+(\tau)| \geq 2$. This implies $|U_0^+(\tau)| > 1$ thanks to the bound (\ref{bound-on-V0}).

If $T < \tau$, then $v_x(t,x) \to -\infty$ for some $x \in \mathbb{R}$ as $t \nearrow T$
by the condition of Lemma \ref{lem-nonlinear-3}. Therefore,
there exists $t_0 \in (0,T)$ such that the bound (\ref{final-bound}) is true.
If $T > \tau$, then the differential equation (\ref{V-U-0}) is valid for $t \in [0,T)$
by Lemmas \ref{lem-nonlinear-3} and \ref{lem-nonlinear-4} so that
the bound (\ref{final-bound}) is true at $t_0 = \tau$ thanks to the bound
$$
\| v_x(t,\cdot) \|_{L^{\infty}} = \| U(t,\cdot) \|_{L^{\infty}} \geq |U_0^+(t)|, \quad t \in [0,T).
$$
In both cases, the bound (\ref{final-bound}) is proven.
\end{proof}

\begin{remark}
The result of Lemma \ref{lem-nonlinear-5} gives the proof of Theorem \ref{theorem-nonlinear} thanks to the
representation (\ref{decomp}) and the characteristic equation (\ref{peak-moves}) at the peak's location at $x = \xi(t)$.
\end{remark}

\section{Conclusion}

We have shown that the passage from linear to nonlinear stability of peakons in $H^1$ is false for the Camassa--Holm
equation because the linear result of Theorem \ref{theorem-linear} gives linearized instability in $H^1$ whereas
the result of Theorem \ref{theorem-CS} gives nonlinear stability of peakons in $H^1$.
On the other hand, we show that the linearized instability in $H^1 \cap C^1_0$ persists
as the nonlinear instability result of Theorem \ref{theorem-nonlinear}. The latter result
is natural for the Camassa--Holm equation where smooth solutions may break in a finite
time with the slopes becoming unbounded from below.

We conclude the paper with possible extensions of our main results.
It is quite natural to prove instability of the peaked periodic waves
with respect to the peaked periodic perturbations in the framework of the CH equation (\ref{CH}).
Furthermore, the same instability is likely to hold for peakons in the
Degasperis--Procesi equation, although the linearized evolution
may not be as simple as the one for the CH equation. Finally, the method of characteristics
is likely to work to prove nonlinear instability of peaked periodic wave
in the reduced Ostrovsky equation (\ref{redOst}), which has been an open problem up to now.

\vspace{0.25cm}

{\bf Acknowledgements.} This project was initiated in collaboration with Y.Liu and G. Gui
during the visit of D.E. Pelinovsky to North-West University at Xi'an in June 2018. DEP thanks
the collaborators for useful comments. The project was completed during the visit of F. Natali to McMaster
University supported by CAPES grant. D.E. Pelinovsky is supported by the NSERC Discovery grant.

\end{document}